\documentclass[reqno]{amsart}
\usepackage{graphicx}
\usepackage{geometry}
\usepackage{amsmath}
\usepackage{amscd}
\usepackage{amsthm}
\usepackage{amsfonts}
\usepackage{amssymb}
\newtheorem{theorem}{Theorem}[section]

\newtheorem{corollary}[theorem]{Corollary}

\newtheorem{lemma}[theorem]{Lemma}
\newtheorem{proposition}[theorem]{Proposition}

\theoremstyle{definition}
\newtheorem{definition}[theorem]{Definition}
\newtheorem{example}[theorem]{Example}

\newtheorem{remark}[theorem]{Remark}

\DeclareMathOperator{\Sp}{Spec}
\DeclareMathOperator{\prim}{Prim}
\DeclareMathOperator{\Hom}{Hom}

\DeclareMathOperator{\GL}{GL}

\begin{document}

\newcommand{\Ga}{\mathbb{G}_a}
\newcommand{\Gap}[1]{\mathbb{G}_{a,#1}}
\newcommand{\alf}[1]{\mathbb{\alpha}_{#1}}
\newcommand{\Zp}{\mathbb{Z}/p\mathbb{Z}}
\newcommand{\gen}[1]{\langle #1 \rangle}
\newcommand{\Fq}{\mathbb{F}_q}
\newcommand{\pa}{\pi}
\newcommand{\Th}{\Theta}
\newcommand{\Thp}{\Theta^{(p)}}
\newcommand{\mx}[1]{\left( \begin{array}{c c} #1 \end{array} \right)}
\newcommand{\them}{them there}

\title{Primitively generated {H}opf orders  in characteristic $p$}
\author{Alan Koch}
\address{Department of Mathematics, Agnes Scott College, 141 E. College Ave., Decatur,
GA\ 30030 USA}
\email{akoch@agnesscott.edu}
\date{\today        }

\begin{abstract}
Let $R$ be a characteristic $p$ discrete valuation ring with field of fractions $K$. Let $H$ be a commutative, cocommutative $K$-Hopf algebra of $p$-power rank which is generated as a $K$-algebra by primitive elements. We construct all of the $R$-Hopf orders of $H$ in $K$; each Hopf order corresponding to a solution to a single matrix equation. For $R$ complete, we give explicit examples of Hopf orders in some rank $p^2$ $K$-Hopf algebras.

\end{abstract}
\maketitle

\section{Introduction}

Let $R$ be a Dedekind domain with field of fractions $K$, and let $H$ be a finite $K$-Hopf algebra. Then an $R$-Hopf order in $K$ is a finitely generated projective $R$-submodule $H_0$ of $H$ which is an $R$-Hopf algebra such that $KH_0=H$, where the coalgebra structure maps are the restrictions of the coalgebra structure maps on $H$. Works of, e.g.,  Byott (\cite{Byott93a},\cite{Byott04}), Childs (\cite{ChildsGreitherMossSauerbergZimmerman98}, \cite{ChildsUnderwood03},\cite{ChildsUnderwood04}, \cite{ChildsSmith05}), Greither (\cite{Greither92}), Larson (\cite{Larson76}), and (especially) Underwood (\cite{Underwood94},\cite{Underwood96}, \cite{UnderwoodChilds06},\cite{Underwood06}, \cite{Underwood08}) have focused on computing $R$-Hopf orders in the case where  $R$ is a characteristic zero complete discrete valuation ring of residue characteristic $p$ containing a $p^{\text{th}}$ root of unity. Typically, these authors have focused on small classes of Hopf algebras (sometimes just a single Hopf algebra), almost always group rings of $p$-power rank. (A notable exception is \cite{Byott93a}, who considers the more general rank $p^2$ case.) Additionally, the author (\cite{KochMalagon07}) has used Larson orders to explicitly describe Hopf orders in elementary abelian group rings; other constructions in (\cite{Koch07a} \cite{Koch12a}) have described orders in abelian group rings of $p$-power rank using Breuil modules (in the former work) and Breuil-Kisin modules (in the latter).

Alternatively, for $G$ a finite flat group scheme defined over $K$, one may be interested in finding models of $G$ over $R$, that is, finite flat group schemes $G_0$ over $R$ such that $G_0 \times_{\Sp R} \Sp K \cong G$. Constructing models for certain group schemes is of interest to many, for example Mezard, Romagny and Tossici (\cite{MezardRomagnyTossici13}, \cite{Romagny12}, \cite{Tossici08}, \cite{Tossici10}). The problem of finding models for these group schemes is the geometric formulation of constructing Hopf orders, and typically the focus of these works is on a small family of group schemes. If $G$ and $G_0$ are affine, say $G=\Sp(H),\;G_0=\Sp(H_0)$, then $H_0$ is an $R$-Hopf order in $H$; in the proper categories of group schemes and Hopf algebras, one sees that these are two different descriptions of the same problem. 

Now suppose $R$ and $K$ have characteristic $p>0$. This paper is an attempt by the author to construct Hopf orders for a larger family of $K$-Hopf algebras, namely the ``primitively generated" Hopf algebras. As one might expect, a finite Hopf algebra is primitively generated if it is generated (as an algebra) by its primitive elements. While many Hopf algebras are not primitively generated, this family is large enough to include $(K\Gamma)^*$ , the linear dual of the group ring, where $\Gamma$ a finite elementary $p$-group; as well as $K[t]/(t^{p^n})$, the Hopf algebra which represents the $n^{\text{th}}$ Frobenius kernel of the additive group scheme.

We construct our Hopf orders by reducing the problem to one of linear algebra. Associated to $H$, a $K$-Hopf algebra of rank $p^n$,  is an $n\times n$ matrix $B$ with entries in $K$ which keeps track of the $p^{\text{th}}$ power of a set of primitive generators for $H$. (This matrix $B$ does depend on the choice of generators, as we shall see.) We choose a matrix $\Th \in \GL_n(K)$ such that $\Th^{-1} B \Thp$ has its entries in $R$, where $\Thp$ is obtained from $\Th$ by raising each entry to the $p^{\text{th}}$ power. Then $\Th^{-1} B \Thp$ is the analogue to $B$ for some $R$-Hopf algebra $H_0$, and $\Th$ describes an embedding of $H_0$ into $H$, thereby realizing $H_0$ as an $R$-Hopf order in $H$. 

Additionally, we will give a simple condition for when two choices of $\Th$ result in the same $R$-Hopf order. Note that this is a different concept from ``isomorphic Hopf algebras", since we can have distinct (as subsets of $H$) Hopf orders which are isomorphic. 

In the case where $R$ is complete, we can impose helpful restrictions on $\Th$. Namely, we can assume $\Th$ is lower triangular, and that the valuation of each entry below the diagonal is bounded by the valuation of the diagonal entry in that row. This allows for a much more explicit computation of Hopf orders, which we do in the case where $H$ has dimension $p^2$.

In \cite{ChildsUnderwood03}, the authors conjecture that ``a classification of Hopf orders [in cyclic group rings] of rank $p^n$ should involve $n(n+1)/2$ parameters: $n$ ‘valuation’ parameters and $n(n-1)/2$ ‘unit parameters’". A similar conjecture holds for Hopf orders in elementary abelian group rings appears in \cite{ChildsGreitherMossSauerbergZimmerman98} and was echoed in \cite{MezardRomagnyTossici13}. By restricting to these lower triangular matrices, the Hopf orders we construct (in the complete case) also depend on $n$ valuation parameters (the diagonal elements) and $n(n-1)/2$ other parameters (not necessarily units). While this similarity is striking, this is also the point: rather than working piecemeal we can consider larger families of $K$-Hopf algebras without losing a lot of specificity. One sees similarities between these results despite the difference in Hopf algebra type as well as characteristic.

A stated above, there are many $K$-Hopf algebras for which the work presented here do not apply, for example group rings of rank $p^n$ for $n \ge 2$. We have chosen to focus on the primitively generated Hopf algebras because of their simplicity: one can use a matrix to completely describe the Hopf algebra. If a more general family of Hopf algebras is desired, it may be possible to use a generalized Dieudonn\'e module theory (see \cite{deJong93}), or characteristic $p$ Breuil-Kisin modules or frames (see \cite{Lau10}), or Falting's strict modules (\cite{Abrashkin06}) . However, along with being more complicated, the correspondence between these modules and their Hopf algebras is not very transparent. By restricting to the primitively generated Hopf algebras, we are allowed to use a Dieudonn\'e module theory for which the correspondence is evident. Additionally, most of the alternative structures above construct Hopf algebras up to isomorphism; care must be taken in cases where a Hopf algebra can have isomorphic Hopf orders. In our work,  Dieudonn\'e modules classify Hopf orders up to isomorphism as well, however given a map between Dieudonn\'e modules the induced map on Hopf algebras can be made very explicit; this will allow us to differentiate between isomorphic, but not equal, Hopf algebras.

The paper is organized as follows. We start by recalling the theory of Dieudonn\'e modules for primitively generated Hopf algebras over an $\mathbb{F}_p$-algebra; typically this theory is described in terms of group schemes rather than Hopf algebras, but we will provide both interpretations. We do not require the $\mathbb{F}_p$-algebra to be a perfect field (or even a field), however it agrees with the classical Dieudonn\'e module theory as developed in \cite{Fontaine75a} if it is. Since the Dieudonn\'e module functor (or functors, since we are interested in two different $\mathbb{F}_p$-algebras, $R$ and $K$) behaves well with base change, we give a condition for when two $R$-Hopf algebras are Hopf orders in the same $K$-Hopf algebra. This condition, that there exists a Dieudonn\'e module map between the respective Dieudonn\'e modules which becomes an isomorphism when base changed to $K$, is precisely the analogue of \cite[2.4.7]{Kisin07} for Breuil-Kisin modules in characteristic zero. Next, we introduce the matrix $\Th$ mentioned above, and show that $\Th$ gives rise to a Hopf algebra if and only if $\Th^{-1}B\Thp\in M_n(R)$. We then restrict to complete discrete valuation rings and show that we may always assume that $\Th$ is of the lower-triangular form mentioned above. Finally, we restrict to Hopf algebras of rank $p^2$ and do four explicit computations, including examples where $H$ is the dual to an elementary group ring, $H$ is a monogenic separable truncated polynomial algebra, and $H$ is a monogenic local truncated polynomial algebra. (Here, the term ``monogenic" refers to an algebra generated by a single element.)

Throughout, all group schemes (with the exception of $\mathbb{G}_a$, below) are affine, commutative, flat, and of $p$-power rank. All Hopf algebras are commutative, cocommutative, projective, and of $p$-power rank. Also, our Hopf algebras $H$ will typically be defined as quotients, e.g. $H=K[t]/(t^p)$; we will, by abuse of notation, refer to elements of $H$ by their coset representative, e.g. $t\in K[t]/(t^p)$. No confusion should arise.

\section{Dieudonn\'{e} Module Theory}

Let $R$ be an $\mathbb{F}_{p}$-algebra, and let $\mathbb{G}_{a,R}$ be the
additive group scheme over $R$. When $R$ is understood, we will simply write
$\mathbb{G}_{a}.$ In this section, we will outline a connection between
certain modules and finite $R$-group schemes $G$ which embed in $\mathbb{G}%
_{a}^{n}$ for some $n$ sufficiently large; such group schemes will be called
{\it additive} since the group operation is inherited from $\Ga^n$, and are necessarily affine and commutative. Note that additive group schemes are finite, and should not be confused with {\it the} additive group scheme $\Ga$.

The correspondence we describe below is well-known -- see, e.g., \cite{deJong93}. 

\subsection{Geometric interpretation}

By considering scalar multiplication and the Frobenius morphism on
$R$-algebras, it is not hard to see that $\operatorname*{End}_{R\text{-Gr}%
}\left(  \mathbb{G}_{a}\right)  \cong R\left[  F\right]  ,$ where $R\left[
F\right]  $ is the noncommutative polynomial ring with $Fa=a^{p}F$ for all
$a\in R.$ Now for any $R$-group scheme $G$ we let $D^{\ast
}(G)=\operatorname*{Hom}\nolimits_{R}\left(  G,\mathbb{G}_{a}\right)  .$ Then
$R\left[  F\right]  $ acts on $D^{\ast}(G)$ via, for $S$ an $R$-algebra,
$\left(  F\cdot f\right)  \left(  s\right)  =F\cdot(f(s)),$ $s\in
S,\;f:G(S)\rightarrow\mathbb{G}_{a}(S).$ This gives $D^{\ast}(G)$ the
structure of an $R[F]$-module, and $D^{\ast}$ is a contravariant functor from
$R$-groups to $R\left[  F\right]  $-modules. This functor does not induce a
categorical anti-equivalence, however if we let AG$_{R}$ denote the category
of finite flat additive groups over $R$, and FF$_{R}$ the category of finite
$R[F]$-modules which are free as $R$-modules, then $D^{\ast}$ does give rise
to a categorical anti-equivalence: AG$_{R}\rightarrow$FF$_{R}$.  Of course, the morphisms in FF$_R$ are the $R$-linear maps which respect the actions of $F$.

Let $M$ be an object in FF$_{R}.$ Since it is free over $R$ we can pick an
$R$-basis $\{e_{1},\dots,e_{n}\}$ for $M.$ Then there exist $a_{j,i}\in R,\;1\leq
i,j\leq n$ such that $Fe_{i}=\sum_{j}a_{j,i}e_{j}.$ For any $R$-algebra $S$ we
define $f_{S}:S^{n}\rightarrow S^{n}$ by%
\[
f_{S}\left(  s_{1},\dots,s_{n}\right)  =\left(  s_{1}^{p}-\sum_{j=1}%
^{n}a_{j,1}s_{j},s_{2}^{p}-\sum_{j=1}^{n}a_{j,2}s_{j},\dots,s_{n}^{p}%
-\sum_{j=1}^{n}a_{j,n}s_{j}\right)  .
\]
Since $S$ has characteristic $p$ one can easily check that $f_{S}$ is a
homomorphism of additive groups. Moreover, this homomorphism is functorial: if
$S_{1},S_{2}$ are $R$-algebras and $g:S_{1}^{n}\rightarrow S_{2}^{n}$ is a
group homomorphism, then $f_{S_{2}}g=gf_{S_{1}}.$ Thus we have an $R$-group
scheme homomorphism $f:\mathbb{G}_{a}^{n}\rightarrow\mathbb{G}_{a}^{n};$ the
kernel of this map is a finite flat group scheme and $D^{\ast}\left(  \ker
f\right)  =M.$

Clearly, the product of any two elements in AG$_R$ is in AG$_R$ and
\[D^*(G_1 \times G_2)=\Hom(G_1\times G_2,\Ga) \cong \Hom(G_1,\Ga)\times \Hom(G_2,\Ga) = D^*(G_1)\times D^*(G_2),\]
hence this functor preserves finite products.

\subsection{Algebraic interpretation}

While the Dieudonn\'{e} correspondence is usually expressed in the literature
geometrically, it will be more natural and useful for us to use the language
of Hopf algebras. 

\begin{definition}
Let $R$ be an $\mathbb{F}_p$-algebra, and let $H$ be an $R$-Hopf algebra. An element $t\in H$ is called {\it primitive} if
\begin{equation}\label{prim}
\Delta(t) = t \otimes_R 1 + 1 \otimes_R t
\end{equation}
where $\Delta$ is the comultiplication on $H$.  The $R$-module of primitive elements is denoted $\prim(H)$. Furthermore, if no proper subalgebra of $H$ contains $\prim(H)$ then $H$ is said to be {\it primitively generated}. \end{definition}

A primitively generated Hopf algebra of rank $p^n$ is generated as an $R$-algebra by a set of (at most) $n$ primitive elements. Equivalently, $H$ is primitively generated if and only if $\Sp(H)$ embeds in $\Gap{R}^N$ for some $N$.

Let PG$_{R}$ denote the category of primitively generated Hopf algebras (of $p$-power rank).  Suppose $H$ is an object in PG$_R$ of rank $p^n$ with a set of primitive generators $\{t_1,\dots,t_n\}$, and let $G=\Sp(H)$. Since $\mathbb{G}_a=\Sp(R[t])$ with $t$ primitive, the projection $R[t_1,\dots,t_n]\to H$ induces a map of group schemes
\[G \to \mathbb{G}_a^n\]
which is injective. Conversely, any finite subgroup of $\mathbb{G}_a$ is represented by a primitively generated Hopf algebra.
Thus, we may define a
covariant functor $D_{\ast}$ from PG$_{R}$ to FF$_{R}$ by
\[
D_{\ast}(H)=D^{\ast}(\operatorname*{Spec}H).
\]
This clearly gives a categorical equivalence. More directly, we may express this as
\[
D_{\ast}(H)=\operatorname{Prim}(H).
\]

Now suppose $M$ is an object of FF$_{R}.$ As in the previous subsection, pick
an $R$-basis $\{e_{1},\dots,e_{n}\}$ for $M\ $and write $Fe_{i}=\sum_{j}%
a_{j,i}e_{j}$ for some $a_{j,i}\in R,\;1\leq i,j\leq n$. Let
\[
H=R[t_{1},\dots,t_{n}]/\{  t_{i}^{p}-\sum_{j=1}^{n}a_{j,i}t_{j}\}
,\;\Delta(t_{i})=t_{1}\otimes1+1\otimes t_{i}.
\]
Then $H$ is a finite abelian Hopf algebra and $D_{\ast}(H)=M.$ Furthermore, it
is evident that 
\[\operatorname{rank}_R H =p^{\operatorname{rank}_{R}M}.\]

In both the geometric and algebraic formulations of Dieudonn\'e modules, we found elements $a_{i,j}$ of $R$ corresponding to $M$. Note that we can just as easily define $M$ by choosing $a_{i,j}$'s. 

\begin{lemma}
For $1\leq i,j\leq n,$ pick $a_{i,j}\in R.$ Set $A=(a_{i,j})  \in
M_{n}(R).$ Let $M=R^{n}$, and make $M$ into an $R[F]$-module via
$Fe_{i}=Ae_{i},$ where $e_{i}$ is the $i^{\text{th}}$ standard basis vector in $R^{n}.$ Then $M$ is an object in FF$_{R},$ corresponding to an object in
PG$_{R}.$
\end{lemma}

If the resulting Hopf algebra above is denoted $H$, we will call the $A$ the matrix {\it associated to} $H$. It is easy to see that different choices of $A$ may result in the same Hopf algebra, hence $A$ is the matrix associated to $H$ for a given choice of primitive generators. With abuse of language, we will write ``$A$ is the matrix associated to $H$", rather than the more accurate ``$A$ is the matrix associated to $H$ given a set of primitive generators for $H$".

\subsection{Some examples}{\label{exsec}}

We now describe the Dieudonn\'{e} module which corresponds to some well-known
group schemes. In addition, we give an example of a Dieudonn\'{e} module which
corresponds to a monogenic Hopf algebra. We will see these examples again later when we construct Hopf orders.

\begin{example}\label{ap}
The unique simple object in the category of (finite flat commutative)
connected unipotent group schemes is the first Frobenius kernel of
$\mathbb{G}_{a}$ -- this group scheme is typically denoted $\mathbf{\alpha
}_{p}$. Recall that $F\in \operatorname*{End}(\mathbb{G}_{a})$ is the Frobenius map, so $\alf{p} = \ker F$. Since $F$ is trivial on $\alf{p}$ it follows that the corresponding Dieudonn\'e module $M$ is a free, rank one $R$-module with $Fm=0$ for all $m\in M$.

For an algebraic interpretation, $\alf{p}=\Sp (H)$ with $H=R[t]/(t^p),\; t\in\prim(H)$. Denoting by $\sigma$ the relative Frobenius map on $H$ we obtain the same module $M$ as above from by observing that $\sigma(t) = 0$. The ``matrix" associated to $H$ is the $1\times 1$ zero matrix.
\end{example}

\begin{example}\label{apn1}
As one of two generalizations of the previous example, we may consider the product $\alf{p}^n =\Sp(R[t_1,\dots,t_n]/(t_1^p,\dots,t_n^p)), \{t_i\}\subset\prim(H)$. Since $D^*$ preserves products, one readily sees that $D^*(\alf{p}^n)=\left(D^*(\alf{p})\right)^n$ is a free $R$-module of rank $n$ with $Fm=0$ for all $m\in M$. Again, the matrix corresponding to $H$ is, of course, the zero matrix.
\end{example}

\begin{example}\label{apn2}
Now consider the $n^{\text{th}}$ Frobenius kernel $F^n$ on $\Ga$, denoted $\alf{p^n}$. Then $\alf{p^n}=\Sp(H)$ with $H=R[t]/(t^{p^n})$ and $t\in\prim(H)$. Alternatively, we may write
\[ H = R[t_1,\dots,t_n]/(t_1^p,t_2^p-t_1\dots,t_n^p-t_{n-1}), t_i\in P(H) \]
by setting $t_i=t^{p^{i-1}}$ for all $i$. Then $\sigma(t_i) = t_{i-1}$ for all $i \ne 1$ and $\sigma(t_1)=0$. Thus the matrix associated to $H$ is
\[A = \left( \begin{array} {c c c c c} 0 & 1 &0 & \cdots  &0 \\
0  & 0  &1 &  \cdots &0\\
\vdots  & \vdots & \ddots & \ddots &\vdots \\
0 & 0  &0 &  \ddots& 1 \\
0 & 0 & 0 &\cdots & 0
\end{array} \right).\]
\end{example}

\begin{example}\label{Zp}
Let $\Zp$ denote the constant group scheme of rank $p$. Since $\Zp=\ker \{(F-1):\Ga \to \Ga\}$ this group scheme is a subgroup of $\Ga$, hence $(\Zp)^n$ is an additive group. Clearly, $Fm=m$ for all $m\in D^*((\Zp)^n)$, hence the matrix which describes the $F$-action on $D^*((\Zp)^n)$ is the $n\times n$ identity matrix.

In terms of Hopf algebras, $\Zp = \Sp (RC_p^*)$, where $RC_p$ is the group ring of rank $p$ and $(-)^*$ indicates duality. Note that we use $C_p$ to denote the cyclic group of order $p$ to differentiate it from the constant group scheme $\Zp$; we will also view $C_p$ as a multiplicative group. 

Typically, elements of the Hopf algebra $RC_p^*$ are described using a basis of orthogonal idempotents. If we write $C_p=\gen{g}$, then $RC_p^*$ is generated as an $R$-vector space by $\{\epsilon_j : 0\le j \le p-1\}$ with $\epsilon_j(x^i)=\delta_{i,j}$ where $\delta_{i,j}$ is the Kronecker delta function. The presence of these orthogonal idempotents show that $RC_p^* \cong R^n$ as $R$-algebras; additionally, $R[t]/(t^p-t) \cong R^n$ as $R$ algebras since $t^p-t$ splits completely. Thus $R[t]/(t^p-t) \cong RC_p^*$ as $R$-algebras, and the image of $t$ under this isomorphism is a primitive element of $RC_p^*$; setting $\Delta(t)=t\otimes 1 + 1 \otimes t$ makes this an isomorphism of $R$-Hopf algebras.

More generally, $(RC_p^*)^n$ can be written as $R[t_1,\dots,t_n]/(t_1^p-t_1,\dots, t_n^p-t_n)$, and the matrix associated to $(RC_p^*)^n$ corresponding to this choice of primitive generating set is the identity $I$, as above.

\end{example}

\begin{example}\label{xp}
Let $G=\Zp \times \alf{p}$. Then $H=R[t_1,t_2]/(t_1^p-t_1,t_2^p)$. The matrix associated to $H$ is
\[A=\left( \begin{array}{c c} 1& 0 \\ 0 & 0\end{array}\right).\]
More generally, the matrix associated to a tensor product of $R$-Hopf algebras is the direct sum of the matrices associated to each of the tensor factors.
\end{example}

\begin{example}\label{mono}
Finally, we give an example where a matrix $A\in M_n(R)$ gives an $R$-Hopf algebra $H$. Pick $n\ge 2$, and let $A$ be the cyclic permutation matrix
\[A = \left( \begin{array} {c c c c c} 0 & 1 &0 & \cdots  &0 \\
0  & 0  &1 &  \cdots &0\\
\vdots  & \vdots & \ddots & \ddots &\vdots \\
0 & 0  &0 &  \ddots& 1 \\
1 & 0 & 0 &\cdots & 0
\end{array} \right).\]

Then $H$ is generated by primitive elements $t_1,\dots,t_n$ with 
\[ t_i^p=\left\{ \begin{array}{c c}t_{i-1} & i >0  \\ t_n & i = 0 \end{array}
 \right. \]

If we set $t = t_n$ then 
\[H = R[t]/(t^{p^n}-t), \]
a monogenic Hopf algebra of rank $p^n$.
The corresponding group scheme is most easily described as $\ker \{ (F^n-1): \Ga \to \Ga\}$. This demonstrates that while a rank $n$ Dieudonn\'e module corresponds to a group scheme which can be embedded in $\Ga^n$, there may be embeddings into fewer copies of $\Ga$.
\end{example}

\section{Hopf Orders}

Throughout this section, let $R$ be a Dedekind domain of characteristic $p$, and let $K$ be its field of fractions. Of particular interest to us will be the cases $K=\Fq(\pa)$, $R=\Fq[\pa]$ and, especially, $K=\Fq((\pa))$, $R=\Fq[[\pa]]$.

\begin{definition}Let $H$ be a $K$-Hopf algebra. Then an {\it $R$-Hopf order in $K$} is a finitely generated projective $R$-submodule $H_0$ of $H$ which is an $R$-Hopf algebra using the operations inherited from $H$, such that $KH_0 = H$. Furthermore, if $H_1$ and $H_2$ are $R$-Hopf orders in the same $K$-Hopf algebra, we say $H_1$ and $H_2$ are {\it generically isomorphic}.
\end{definition}

 Using geometric language, $\Sp(H_0)$ is said to be a {\it model} for $\Sp(H)$. In fact, the term ``generically isomorphic" is borrowed from the geometric language, since saying that $\Sp(H_1)$ and $\Sp(H_2)$ are models of the same $K$-group scheme means that they have isomorphic generic fibers.

Note that in some literature $KH_0$ is expressed as $H_0\otimes_R K$. We prefer the former notation since it is imperative that we view a Hopf order as an $R$-Hopf algebra which is a subset of $H$, not one which embeds in $H$. We will see different Hopf orders which are isomorphic as Hopf algebras, so we adopt a notation which distinguishes between such objects.

Clearly, an $R$-Hopf order $H_0$ in a primitively generated $K$-Hopf algebra $H$ is primitively generated, however the converse is also true. If $G=\Sp(H)$ then $G$ embeds in some $\Gap{K}^N$. Thus, any model of $G$ embeds in a model of $\Gap{K}^N$ with connected fibers. By \cite[Cor. 3.4]{WeisfeilerDolgacev74}, the only such model of $\Gap{K}^N$ is $\Gap{R}^N$, hence the Hopf order is primitively generated.

Since $R$ and $K$ are both $\mathbb{F}_p$-algebras, each has a theory of Dieudonn\'e modules. The following shows that the modules respect the change of scalars from $R$ to $K$.

\begin{lemma}
Let $D_{*,R}$ (resp. $D_{*,K}$) denote the Dieudonn\'e module correspondence PG$_R \to $FF$_R$ (resp. PG$_K \to $FF$_K$). Let $G_0$ be an additive $R$-group scheme, and write $G=\Sp(H_0)$. Then 
\[KD_{*,R}(H_0)  = D_{*,K}(KH_0).\]
\end{lemma}

Of course, this result could have been described in terms of the contravariant Dieudonn\'e module functor.

\begin{proof}
Pick an $R$-basis $\{t_1,\dots,t_n\}$ for $\prim(H_0)$. Since $\sigma(t_i)\in \prim(H_0)$ for all $i$ we can write
\[\sigma(t_i)=\sum_{j=1}^n a_{j,i} t_j \]
for some $a_{j,i} \in R.$ Let $A_0=(a_{i,j})\in M_n(R)$. Then $M_0=D_{*,R}(H_0)$ is $R$-free of rank $n$, and we can write $M_0=\oplus_{i=1}^n Re_i$,  such that $Fe_i = A_0 e_i$. 
Now let $H=KH_0$. Since $\{t_1,\dots,t_n\}$ is a $K$-basis for $\prim(H)$ and the value of $\sigma(t_i)$ does not change, $M:=D_{*,K}(H)$ is a rank $n$ module over $K$ with $Fe_i = A_0 e_i$. Therefore, $M=KM_0$.
\end{proof}

\begin{remark}
In the proof of this result, all that was used was that $R$ is contained in $K$. More generally, both $D_{*}$ and $D^*$ behave very well under base change.
\end{remark}

When necessary to indicate the base ring, we will continue to use the notations $D_{*,R}$ and $D_{*,K}$.

The next result is the analogue of \cite[2.4.7]{Kisin07} mentioned in the introduction, as it indicates when two $R$-Hopf algebras are generically isomorphic.

\begin{lemma} \label{invert}
Let $H_1$ and $H_2$ be primitively generated $R$-Hopf algebras. Let $M_1,M_2$ be the corresponding Dieudonn\'e modules, and suppose $f:M_1\to M_2$ is an $R[F]$-module map. If the induced $K[F]$-module map $KM_1 \to KM_2$ is an isomorphism, then $D_*(f)(H_1)$ and $H_2$ are Hopf orders in $KH_2$.
\end{lemma}

Loosely, we say that $H_1$ and $H_2$ are orders in the same Hopf algebra. However, strictly speaking this is not necessarily the case: we need to first embed $H_1$ into $H_2$, which we do here via $D_*(f)$.

\begin{proof}
If $KM_1 \to KM_2$ is an isomorphism then $f$ is necessarily injective, hence so is $D_*(f)$. Since
\[KD_{*,R}(H_1)=D_{*,K}(KH_1))\cong D_{*,K}(D_*(f)(KH_1))=D_{*,K}(KH_2))=KD_{*,R}(H_2),\]
it follows that $H_1$ and $H_2$ are generically isomorphic. The map $D_*(f)$ gives the embedding $H_1\hookrightarrow H_2\subset KH_2$.
\end{proof}

\begin{remark}
The lemma above has a converse: if $KH_1\cong KH_2$ then there exists an $R[F]$-module map $f:M_1\to M_2$, with $M_i=D_{*,R}(H_i)$. To see this, pick $R$-bases $\{t_1,\dots,t_n\}$ and $\{u_1,\dots,u_n\}$ for $H_1$ and $H_2$ respectively. Then these sets are also $K$-bases for $KH_1$ and $KH_2$ respectively. Let $\phi: KM_1\to KM_2$ be an isomorphism of $K[F]$-modules (which we know exists since the $K$- Hopf algebras are isomorphic). Then there exist $c_{j,i}\in K$ such that
\[\phi(t_i)=\sum_{j=1}^n c_{j,i}u_j,\; 1\le i \le n.\]
Let $m=\min\{v(c_{j,i}):1\le i,j\le n\}$, and define $\psi:KM_1\to KM_2$ by
\[\psi(t_i)=\sum_{j=1}^n \pa^{-m} c_{j,i}u_j,\; 1\le i \le n.\]
Then $\psi$ restricts to a an injective map $f:M_1 \to M_2$.
\end{remark}

\section{A Linear Algebra Construction}

We have seen that Dieudonn\'e modules can be described using matrices. We will now see how those matrices allow us to find Hopf orders as well. In this section, $R$ is a discrete valuation ring (not necessarily complete) with uniformizing parameter $\pa$. We let $v_K$ denote the normalized valuation on $K$. 

Let us fix a primitively generated $K$-Hopf algebra $H$ of rank $p^n$, and let $A\in M_n(K)$ be the matrix associated to $H$ with respect to some primitive generating set $\{t_1,\dots,t_n\}$ of $H$. Furthermore, by replacing $\{t_1,\dots,t_n\}$ with $\{\pa^v t_1,\dots,\pa^v t_n\}$ for $v$ suitably large we may assume that $A\in M_n(R)$. Let $H_1$ be the $R$-Hopf algebra which has $A$ associated to it. Then $H_1$ is an $R$-Hopf order of $H$.

Now suppose that $H_2$ is another $R$-Hopf order in $H$. Then $H_1$ and $H_2$ are generically isomorphic, hence there is a $K[F]$-module isomorphism $\Theta: M_1 \to M_2$, where $D_{*,K}(H_i) = M_i, i=1,2$.
For $\Theta$ to be a $K[F]$-module map, it needs to be $K$-linear and respect the action of $F$. Then $\Theta(Fe_i)=F\Theta(e_i)$ for all $1\le i\le n$, where $\{e_i:1\le i\le n\}$ is a basis for $M_1$.

Write $M_2=\oplus_{i=1}^n Rf_i$, $Ff_i=\sum_{j=1}^n b_{j,i}f_j, b_{j,i}\in R$, and write $\Theta(e_i)=\sum_{j=1}^n\theta_{j,i} f_j$. Then
\begin{align*}
\Theta(Fe_i) & = \Theta(\sum_{j=1}^n a_{j,i}e_j) \\
	&= \sum_{j=1}^n a_{j,i} \Theta(e_j) \\
	&= \sum_{j=1}^n \sum_{k=1}^n a_{j,i} \theta_{k,j}f_k \\
	&=\sum_{k=1}^n \Big(\sum_{j=1}^n \theta_{k,j} a_{j,i} \Big) f_k
\end{align*}
and
\begin{align*}
F(\Theta(e_i)) & = F(\sum_{j=1}^n\theta_{j,i} f_j) \\
	&= \sum_{j=1}^n\theta_{j,i}^pF(f_j)\\
	&=\sum_{j=1}^n\sum_{k=1}^n \theta_{j,i}^p b_{k,j} f_k\\
	&=\sum_{k=1}^n\Big(\sum_{j=1}^n  b_{k,j} \theta_{j,i}^p\Big)  f_k
\end{align*}

If we denote the matrix which represents the $K$-linear map $\Th$ by $\Th$ as well, and if $B=(b_{i,j})$ then 
\begin{equation}\label{id}
\Th A = B \Thp.
\end{equation}

In fact, we have

\begin{proposition}\label{thetaprop}
Let $H_1$ and $H_2$ be primitively generated $R$-Hopf algebras of rank $p^n$. Let $A\in M_n(R)$ be the matrix associated to $H_1$, $B\in M_n(R)$ the matrix associated to $H_2$. Then $H_1$ and $H_2$ are generically isomorphic if any only if there exists a $\Th=(\theta_{i,j})\in GL_n(K)$ such that $\Th A = B \Thp$.
\end{proposition}

\begin{proof}
If $H_1$ and $H_2$ are generically isomorphic, then the isomorphism $KH_1 \to KH_2$ induces a map $\Th$ on Dieudonn\'e modules which satisfies equation (\ref{id}) above. As $\Th$ is invertible, the matrix representing it must be in $\GL_n(K).$ 

Conversely, given such a $\Th$ we have an isomorphism $D_{*,K}(KH_1)\to D_{*,K}(KH_2)$.
\end{proof}

The full result is as follows.
\begin{theorem}\label{main}
Let $H$ be a primitively generated $K$-Hopf algebra, and let $B \in M_n(R)$ be the matrix associated to it. Pick $A = (a_{i,j}) \in M_n(R)$. Suppose there exists a $\Th\in \GL_n(K)$ such that  $\Th A = B \Thp$. Let
\[H_0=R[u_{1},\dots,u_{n}]/\{  u_{i}^{p}-\sum_{j=1}^{n}a_{j,i}u_{j}\}
,\;\Delta(u_{i})=u_{1}\otimes1+1\otimes u_{i}.
\]
Then $H_0$ can be embedded in $H$ as an $R$-Hopf order. Furthermore:
\begin{enumerate}
\item $H_0$ embeds in $H$ via $\Th$, i.e.
\[H_0 = R \left[ \left\{ \sum_{j=1}^{n}\theta_{j,i}t_j : 1\le i \le n \right\}\right] \subset H \]
\item If $A,A' \in M_n(R)$ are the matrices associated to the Hopf orders $H_0,H_0'$ of $H$ via the embeddings $\Th,\Th'$ respectively, then the $H_0=H_0'$ if and only if $\Th^{-1}\Th' \in M_2(R)^{\times}$.
\end{enumerate}
\end{theorem}

\begin{proof}
That $KH_0 \cong H$ follows from Proposition \ref{thetaprop}, and is given on the Dieudonn\'e modules by $\Th$. Of course, the Dieudonn\'e modules are the primitive elements of the $K$-Hopf algebras. Since each Hopf algebra is generated (as an algebra) by its primitive elements, the Hopf algebra isomorphism $KH_0 \to H$ above can be completely described by the images of the $u_i$, which are given by $\Th$, explicitly:
\[u_i \mapsto \sum_{j=1}^{n}\theta_{j,i}t_j.\]
 By restricting the domain to $H_0$ we get the  embedding as described in (1). 

For (2), if $\Th^{-1}\Th' \in M_2(R)^{\times}$, then $\Th' = \Th U$ for some matrix $U$ invertible in $R$. The embeddings of $H_0$ and $H_0'$ in $H$ are

\[H_0 = R[\{ \sum_{j=1}^{n}\theta_{j,i}t_j : 1\le i \le n \}] \text{ and }  
H_0' = R [\{ \sum_{j=1}^{n}\theta_{j,i}'t_j : 1\le i \le n \}],\]
and $U$ induces an invertible change of coordinates map $H_0'\to H_0$.
\end{proof}

We expand upon an observation from the proof of (2) above and restate it, since it will be very useful in Section {\ref{exes}}.

\begin{corollary}\label{maincor}
With the notation above $A$ and $A'$ are associated to the same Hopf order if and only if $\Th' = \Th U$ for some matrix $U$ invertible in $R$. 
\end{corollary}

\begin{proof}
We have seen that if $\Th' = \Th U$ then the same Hopf order is produced. Now suppose that we have an $A$ and $\Th$ which give a Hopf order, and we pick $U$ to be any element of $M_n(R)^{\times}$. Set $\Th' = \Th U$. The proof will be complete provided that we can show $(\Th')^{-1}(\Th')^{(p)} \in M_n(R).$ Since $U^{-1},(\Th^{-1}\Thp)$, and $U^{(p)} \in M_n(R)$,
\[ (\Th')^{-1}(\Th')^{(p)} = (\Th U)^{-1} (\Thp U^{(p)}) = U^{-1}( \Th^{-1}\Thp) U^{(p)}\in M_n(R). \]
\end{proof}

\subsection{The significance of $A$} In the construction above, the matrix $\Th$ gives the embedding. As we will see below, we will construct Hopf orders by picking $\Th$ first, then letting $A=\Th^{-1} B \Thp$; if $A$ has its entries in $R$ then we will have constructed a Hopf order. From this perspective, the entries of $A$, aside from whether or not they are in $R$, appear irrelevant. However, $A$ carries important information about the Hopf order as an $R$-Hopf algebra, and can be used to quickly obtain certain interesting characteristics of $A$.

For example, let $H$ be the Hopf algebra
\[H = K[t_1,t_2,t_3]/(t_1^p-\pa^3 t_2, t_2^p - \pa^2 t_1, t_3^p-\pa^4t_3). \]
Let 
\[ \Th = \left( \begin{array}{c c c}\pa & 0 & 0 \\ 1 & 1 & 0 \\ 1 & 0 & \pa \end{array} \right) . \]
Then
\[ A = \Th^{-1} \left( \begin{array}{c c c}0 & \pa^2 & 0 \\ \pa^3 & 0 & 0 \\ 0 & 0 & \pa^4 \end{array} \right) \Thp =  
\left( \begin{array}{c c c}\pa & \pa & 0 \\ \pa^{p+3} - \pa & -\pa & 0 \\ \pa^3-1 & -1 & \pa^{p+3} \end{array} \right) , \]
giving the Hopf order
\[ H_0:=R[\pa t_1 + t_2 + t_3, t_2 ,\pa t_3] \subset H.\]

As an $R$-Hopf algebra, we have
\[H_0 = R[u_1,u_2,u_3]/(u_1^p - \pa u_1 - (\pa^{p+3}-\pa)u_2 - (\pa^3 - 1)u_3, u_2^p - \pa u_1 +\pa u_2 +u_3, u_3^p - \pa^{p+3} u_3).\]

Using the geometric language of models, the additive $R$-group scheme $G_0=\Sp(H_0)$, explicitly given by
\[G_0(S) = \{ (s_1,s_2,s_3) \in S^3: s_1^p = \pa s_1 + (\pa^{p+3}-\pa)s_2 + (\pa^3 - 1)s_3,s_2^p=\pa s_1 -\pa s_2 -s_3,s_3^p + \pa^{p+3} s_3 \} \]
for $S$ an $R$-algebra, is a model of $G=\Sp(H)$.

Furthermore, since the theory of Dieudonn\'e modules is fully compatible with the theory over perfect fields, the matrix $A$, when reduced mod $\pa$ gives the algebra structure of $H_0\otimes_R \Fq$, where $\Fq$ is the residue field of $R$. Thus,
\[H_0\otimes_R \Fq \cong \Fq[u_1,u_2,u_3]/(u_1^p+u_3,u_2^p+u_3,u_3^p)\cong \Fq[u_1,y]/(u_1^{p^2},y^p) \]
where $y=u_1-u_2$.  Geometrically, we say that $\Sp (H_0)$ has special fibre isomorphic to $\alf{p^2}\times\alf{p}$; algebraically, we see that $H_0$ is a local-local Hopf algebra, that is, a local $R$-algebra with local dual. Generally, $\Th \bmod \pa$ gives us, e.g., a minimal set of algebra generators of $H_0\otimes_R \Fq$, hence of $H_0$ by Nakayama's Lemma.

\section{Rank $p$ Hopf Orders}

We illustrate how our technique can be used to find Hopf orders in rank $p$ Hopf algebras. While identifying these Hopf orders is nothing new thanks to \cite{TateOort70}, our desire here is to show how the results can be obtained quickly using Dieudonn\'e modules.
Of course, in the rank $p$ case, $n=1$, so our ``matrices" are simply elements of either $K$ or $R$.

We will continue to allow $R$ to be any discrete valuation ring of characteristic $p$, however in the case where $R$ is complete we obtain a nicer description of the Hopf orders.

Let us first describe all rank one Dieudonn\'e modules over $K$. They are all of the form $M_b = Ke$ with $Fe=be$ for some $b\in K$. Also, $M_a \cong M_b$ if and only if there is a $\theta \in K^{\times}$ such that $\theta a = b\theta^p$. We split our study into two natural pieces.

\subsection{Local case} Let $b=0$. Then the corresponding Hopf algebra is $H=K[t]/(t^p)$, i.e., $\Sp(H)=\alf{p}$, and is the unique monogenic local Hopf algebra with local dual of rank $p$. Clearly, if $\theta a = b\theta^p$ then $a=0$ as well, so there is a single choice of $b$ which produces this $K$-Hopf algebra. However, if $a=b=0$ then {\it any} nonzero $\theta$ satisfies the equality above. Each choice of $\theta$ produces a (potentially) different Hopf order, despite the fact that the orders are all isomorphic as Hopf algebras. The Hopf orders we construct are of the form
\[H_{\theta} = R[\theta t],\]
and $H_{\theta} = H_{\theta'}$ if and only if $\theta '/\theta \in R^{\times}$. In particular, if $R$ is complete, then the Hopf algebras are parameterized by the valuations of $\theta$, hence the $R$-Hopf orders in the complete case are
\[H_i = R[\pa^i t], i\in \mathbb{Z}.\]

\subsection{Separable case} Now we suppose $b\ne 0$. Then $H=K[t]/(t^p-bt)$ is separable. Let $\theta\in  K^{\times}$ and let $a=b\theta^{p-1}$. Then $M_a \cong M_b$, so the isomorphism classes of separable $K$-Hopf algebras of rank $p$ are parameterized by $K^{\times}/(K^{\times})^p$ (as is well-known); furthermore, $H\otimes_K K^{p^{-\infty}} \cong K^{p^{-\infty}}C_p^*$, where $K^{p^{-\infty}}$ is the separable closure of $K$. Thus, $H$ is a form of $KC_p^{\ast}$; in the language of schemes, $\Sp H$ is geometrically isomorphic to $\Zp$.

Let us find the Hopf orders in $H=K[t]/(t^p-bt)$. By replacing $t$ with $\pa^vt$ for some suitably chosen $v$ assume $0\le v_K(b) \le p-2$. Pick $\theta \in K^{\times}$ such that $a=b\theta^{p-1}\in R$. Since  $v_K(a) = v_K(b)+(p-1)v_K(\theta)$ we see that $v_K(a) \ge 0$ if and only if $v_K(\theta) \ge 0$. Thus, the Hopf orders are of the form
\[H_{\theta} = R[\theta t],\; \theta \in R\]
and as before $H_{\theta} = H_{\theta'}$ if and only if $\theta '/\theta \in R^{\times}$. Thus, if $R$ is complete, then the Hopf orders are
\[H_i=R[\pa^i t],\;i\ge 0.\]

Notice that this is in stark contrast to the Hopf orders in the cyclic group ring $KC_p$: orders in $KC_p:=K\gen{g}$ are of the form $R[\pa^{-i}t],\;i \ge 0$, where $t=g-1$. This contrast is to be expected: for example, \cite[5.2]{Childs00} shows that $RC_p$ is the unique minimal Hopf order in $KC_p$ by showing that $RC_p^*$ is the maximal Hopf order in $KC_p^*$ and applying a duality argument. 

\section{Higher Rank Hopf Orders}\label{exes}

We now turn to the more ambitious task of finding Hopf orders in rank $p^n$ Hopf algebras. In this section, we assume that $R$ is complete.

\subsection{Reducing the choices of $\Th$}

It turns out that we can assume a nice form for the matrix $\Th$ which gives the embedding.

\begin{proposition}

Let $H$ be a primitively generated $K$-Hopf algebra of rank $p^n$, and let $B$ be the matrix associated to $H$. Let $H_0$ be an $R$-Hopf order in $H$, and let $A$ be the matrix associated to $H_0$. Then there exists $\Theta=(\theta_{i,j}) \in GL_n(K)$ such that
\begin{enumerate}
\item $\Th$ is lower triangular,
\item $v_K(\theta_{i,i})\ge v_K(\theta_{i,j}),\; j \le i \le n$,
\item $\theta_{i,i}=\pa^{-v_K(\theta_{i,i})}$,
\item $\Th A = B\Thp$.
\end{enumerate}
\end{proposition}

\begin{proof}
Of course, by Theorem \ref{main} we know that there exists a $\Th$ such that (4) holds. However, Corollary \ref{maincor} shows that the choice of $\Th$ is not unique: we may replace $\Th$ by $\Th U$ for some $U\in M_2(R)^{\times}$. We will make a strategic choice of $U$ which makes (1) through (3) hold above. The matrix $U$ will be constructed in several steps.

First, we wish to order the columns so that $v_K(\theta_{1,i}) \le v_K(\theta_{1,j})$ for all $1\le  j \le n$.  Pick $1\le m \le n$ such that $v_K(\theta_{1,m})   \le v_K(\theta_{1,j})$ for all $1 \le j \le n$. Let $U$ be the elementary matrix obtained from the identity matrix by swapping the first and $m^{\text{th}}$ columns. Then $\Th U$ is computed by swapping the same two columns in $\Th$.

Next, we wish to make $\theta_{1,j}=0$ for all $j>1$. This is accomplished as follows. Define
\[ u_{i,j}' = \left\{\begin{array}{c c} -\theta_{1,j}/\theta_{1,1} & i=1<j  \\ 0 & \text{otherwise} \end{array} \right. \]
and let $U=I+(u_{i,j}')$.  Note that $v_K(u_{i,j}') \ge 0$, so $U \in M_n(R)^{\times}$. Then $\Th U$ is a new matrix $\Th$ after the appropriate column operations have been done to zero out all entries to the right of $\theta_{1,1}$. 

Now suppose that for each $i<k\le n$ we have $\theta_{i,j}=0$ for all $j>i$.  Pick $k \le m \le n$ such that $v_K(\theta_{k,m})   \le v_K(\theta_{k,j})$ for all $k \le j \le n$. Let $U$ be the elementary matrix obtained from the identity matrix by swapping the $k^{\text{th}}$ and $m^{\text{th}}$ columns, and replace $\Th$ with $\Th U$. Then let 
\[ u_{i,j}' = \left\{\begin{array}{c c} -\theta_{k,j}/\theta_{k,k} & i=k<j  \\ 0 & \text{otherwise} \end{array} \right. \]
and let $U=I+(u_{i,j}')$. Notice that $\Th$ and $\Th U$ agree on the first $k-1$ rows.  Replacing $\Th$ by $\Th U$ results in a matrix with zeros in the $k^{\text{th}}$ row to the right of $\theta_{k,k}$. Thus $\Th$ is now lower triangular, so (1) has been shown to hold.

To obtain (2), we again work inductively on the rows. The first row clearly has this property; we now suppose that  $v_K(\theta_{i,i})\ge v_K(\theta_{i,j}),\; j \le i < k \le n$. Let
\[ u_{i,j}' = \left\{\begin{array}{c c} 1 & i=k<j, \;  v_K(\theta_{k,j}) > v_K(\theta_{k,k}) \\ 0 & \text{otherwise} \end{array} \right. \]
and let $U=I+(u_{i,j}')$. In other words $U$ is a lower triangular matrix with $1$'s on the diagonal and $1$'s where the valuation of the corresponding entry in $\Th$ exceeds the valuation of the diagonal element on its row. Again, $\Th$ and $\Th U$ agree on the first $k-1$ rows. Also, the $(k,j)^{\text{th}}$ entry of this product is
\[ \left\{ \begin{array}{c c}\theta_{k,k}+\sum \theta_{k,\ell} &  k<j, \;  v_K(\theta_{k,j}) > v_K(\theta_{k,k}) \\ \theta_{k,j} & \text{ otherwise } \end{array} \right.\]
where the sum is over all $\ell$ such that $v_K(\theta_{k,\ell}) > v_K(\theta_{k,k})$. This clearly does not change the valuation if $v_K(\theta_{k,j})<v_K(\theta_{k,k})$; for the other terms the valuation is now the same as $v_K(\theta_{k,k})$. This shows that (2) holds. Since $U$ is lower triangular, so is $\Th U$ hence (1) is preserved. 

For (3), simply replace $\Th$ with $\Th U$, where $U=(u_{i,j})$ is a diagonal matrix with 
\[u_{i,i} = \pa^{-v_K(\theta_{i,i})}\theta_{i,i} \in R^{\times}.\]
Note that this change of $\Th$ does not affect the valuation of the entries, so (1) and (2) still hold.
\end{proof}
\begin{definition}
We will call a matrix $\Th \in \GL_2(K)$ satisfying (1) - (3) above a {\it diagonal dominant lower triangular} matrix, or DDL matrix for short.
\end{definition}

\begin{remark}

Under this construction, $\Th$ cannot be a diagonal matrix. A consequence of this is that none of the Hopf orders we construct will be presented in a ``diagonal" form such as
\begin{equation}\label{Larson}
R[\pa^{i_1}t_1,\dots, \pa^{i_n}t_n].
\end{equation}

We proceed this way for two reasons. First, it is much simpler to consider one family of matrices (DDL matrices) than two (``diagonal strictly dominant lower triangular" and ``diagonal"). Second, orders presented in the form (\ref{Larson}) look very much like Larson's constructions in \cite{Larson76}, where he finds some orders in group rings. We do not wish to present a parallel class of Hopf orders here since we feel that to do so would be artificial. There is nothing inherently interesting about ``Larson-like" Hopf orders here since they depend on a choice of basis for $H$: by replacing parameters $t_1,\dots, t_n$ with another set $u_1,\dots, u_n$ can make non-Larson-like Hopf orders Larson-like and vice versa. In \cite{Larson76}, the analogue of the $t_i$ -- namely, $g_i - 1$ where the $g_i$ generate the group -- is fixed, and certainly is a distinguished basis in the group ring. In our theory, there is no such distinguished basis. 

In fact, we do get the Hopf orders of the form  (\ref{Larson}), they are simply presented differently, for example 
\[R[t_1,t_2] = R[t_1+t_2,t_2].\]

\end{remark}

The general strategy we will employ to find Hopf orders is as follows. Pick a DDL matrix $\Th$ and let $A=\Th^{-1}B\Thp$. Then a Hopf order is constructed if and only if $A \in M_n(R)$. 

\subsection {The case $n=2$} For Hopf orders of rank $p^2$, we can be more explicit. Let us write
\begin{equation}\label{form22}
\Th = \mx{\pa^i & 0 \\ \theta & \pa^j},\;v_K(\theta)\le j.
\end{equation}

As is to be expected, different DDL matrices do not necessarily produce different Hopf orders. The following result makes it easy to identify isomorphism classes.

\begin{proposition}
Let 
\[\Th = \mx{\pa^i & 0 \\ \theta & \pa^j},\;\Th' = \mx{\pa^{i'} & 0 \\ \theta' & \pa^{j'}},\; v_K(\theta)\le j,v_K(\theta')\le j'.\]
Then $\Th$ and $\Th'$ give the same Hopf order if and only if $i=i',\;j=j'$, and \[v_K(\theta-\theta') \ge j.\]
\end{proposition}

\begin{proof} We know that $\Th$ and $\Th'$ give the same Hopf order if and only if $\Th^{-1}\Th'\in M_2(R)^{\times}$, so we compute:
\[\Th^{-1}\Th' = \mx{\pa^{i'-i} & 0 \\ \pa^{-j}\theta'-\pa^{i'-(i+j)}\theta & \pa^{j'-j}} \]
This matrix is an invertible matrix in $M_2(R)$ if and only if $i=i', j=j'$, and
\[ \pa^{-j}\theta'-\pa^{i-(i+j)}\theta=T^{-j}(\theta'-\theta)\in R,\]
from which the result follows.
\end{proof}

This allows for a very explicit description of the Hopf orders inside of a given rank $p^2$ Hopf algebra.

\begin{corollary}
Let $B$ be the matrix associated to a primitively generated $K$-Hopf algebra. Let $\Th$ be a DDL matrix as in equation (\ref{form22}) such that $\Th^{-1}B\Thp\in M_n(R)$. Then the Hopf order corresponding to $\Th$ is
\begin{equation}\label{niceform}
H_{i,j,\theta}= R[\pa^it_1 + \theta t_2, \pa^j t_2],
\end{equation}
and $H_{i,j,\theta}=H_{i,j,\theta'}$ if and only if $v_K(\theta-\theta') \ge j$.
\end{corollary}

Notice that $H_{i,j,\theta}$ can be simplified to a diagonal form if and only if $v_K(\theta) = j$.

\begin{remark}
One may be tempted to simply pick $i,j,\theta$ and create the expression $H_{i,j,\theta}$ in (\ref{niceform}) directly. However, that $\Th^{-1}B\Thp\in M_n(R)$ always needs to be verified. For any $i,j,\theta$, while $H_{i,j,\theta}$ is closed under comultiplication  and $KH_{i,j,\theta}=H$, the equation $\Th^{-1}B\Thp\in M_n(R)$ holds if any only if $H_{i,j,\theta}$ is finitely generated.
\end{remark}

\subsection{Some examples}
We now return to the examples introduced in section {\ref{exsec}}, the exception being Example \ref{ap}, which was considered in the previous section. We will look at each example in the rank $p^2$ case, although the first example below will also be done for general rank $p^n$.

\begin{example} [\ref{apn1}]
Let $H=K[t_1,\dots,t_n]/(t_1^p,\dots,t_n^p)$, i.e., $\Sp(H)=\alf{p}^n.$ The matrix associated to $H$ is $B=0$. Thus, $\Th A = B\Thp$ if and only if $A=0$. Therefore, each $\Th\in\GL_n(K)$ gives a Hopf order. The set of all Hopf orders in this case can be parameterized by the set of left cosets \[\GL_n(K) / M_n(R)^{\times}.\]
When $n=2$ we have $H_{i,j,\theta}= R[\pa^it_1 + \theta t_2, \pa^j t_2]$ is a Hopf order for all choices of $i,\;j$, and $\theta$. To eliminate equal Hopf orders, we can require $\theta = \pa^j$ or $\deg(\theta) < j$, at which point $v_K(\theta) \le j$ is not needed as an explicit restriction. 
\end{example}

\begin{example}[\ref{apn2}]\label{apn2too}
Let $H=K[t]/(t^{p^2})$, so $\Sp(H)=\alf{p^2}.$ Letting $t_1=t^p,\;t_2= t$ gives $H=K[t_1,t_2]/(t_1^p,t_2^p-t_1)$, and the matrix $B$ is
\[B =\mx{0 & 1 \\ 0 & 0}.\]
Let
\begin{align*}
A = \Th^{-1}B\Thp &= \mx{\pa^i & 0 \\ \theta & \pa^j}^{-1} \mx{0 & 1 \\ 0 & 0} \mx{\pa^{pi} & 0 \\ \theta^p & \pa^{pj}} \\
&=\mx{\pa^{-i}\theta^p & \pa^{pj-i} \\ -\pa^{-(i+j)}\theta^{p+1} & -\pa^{(p-1)j-i}\theta}.
\end{align*}
If $A\in M_2(R)$, then clearly $pj\ge i$. Additionally,
\begin{equation}\label{apn2val}
v_K(\theta) \ge \max\{i/p,(i+j)/(p+1), i-(p-1)j \}. 
\end{equation}
But, since $i-pj \le 0$,
\begin{align*}
i-pj &\le \frac{i-pj}p \\
i-(p-1)j &\le j+\frac{i-pj}p =\frac{i}{p},\\
\end{align*}
and
\begin{align*}
i- pj &\le 0  \\
pi- p^2j &\le 0 \\
(p+1)i- (p^2-1)j &\le i+ j\\
  i-(p-1)j&\le\frac{i+j}{p+1} ,
\end{align*}
so we can simplify the restriction on $v_K(\theta)$ significantly, obtaining
\[H_{i,j,\theta}=R[\pa^it_1+\theta t_2, \pa^j t_2] = R[\pa^it^p+\theta t, \pa^j t],\;pj \ge i,\; i-(p-1)j\le v_K(\theta)\le j.\]
If $i \le pj$ then $i-(p-1)j \le j$, so for any choice of $i,j$ with $i \le pj$ there will exist choices of $\theta$ which will produce Hopf orders

Note that $H_{i,j,\theta}$ is monogenic if and only if $v_K(\theta) = j$ and $pj=i$.
\end{example}

\begin{example}[\ref{xp}]
Let $H=R[t_1,t_2]/(t_1^p-t_1,t_2^p)$, so $\Sp(H)=\Zp\times\alf{p}$. The matrix associated to $H$ is 
\[B=\mx{1 & 0 \\ 0 & 0}.\]
Let
\begin{align*}
A = \Th^{-1}B\Thp &= \mx{\pa^i & 0 \\ \theta & \pa^j}^{-1} \mx{1 & 0 \\ 0 & 0} \mx{\pa^{pi} & 0 \\ \theta^p & \pa^{pj}} \\
&=\mx{\pa^{(p-1)i} & 0 \\ -\pa^{(p-1)i-j}\theta & 0}.
\end{align*}
For $A$ to be in $ M_2(R)$, we require $i\ge 0$ and 
\[j-(p-1)i\le v_K(\theta) \le j,\] so the Hopf orders are
\[H_{i,j,\theta} = R[\pa^it_1+\theta t_2,\pa^j t_2],\;[j-(p-1)i\le v_K(\theta) \le j.\]
\end{example}

In particular, note that  $R[t_1,t_2]/(t_1^p-t_1,t_2^p)\cong R[t]/(t^p-t)\otimes R[t]/(t^p)$, and that the example above shows that there are Hopf orders in tensor products which are not tensor products of Hopf orders. This type of behavior has been noted before, for example in \cite{ChildsGreitherMossSauerbergZimmerman98} it is evident that most Hopf orders in $KC_p^n,\;n\ge 2$ are not Larson orders. A difference between our example and examples of the form $KC_p^n$, of course, is that our tensor factors are different from each other. Despite a perceived lack of compatibility between $R[t]/(t^p-t)$ and $R[t]/(t^p)$, most of the Hopf orders we construct --the ones with $v_K(\theta)<j$ -- involve some interaction between the two Hopf algebras.

\begin{example}[\ref{Zp}]
We return to the Hopf algebra  $(RC_p^*)^2=R[t_1,t_2]/(t_1^p-t_1,t_2^p-t_2)$, which represents the constant elementary abelian group scheme of order $p^2$. In this case, $B=I$ so we need to determine whether $\Th^{-1}\Thp\in M_2(R)$. We have
\[\Th^{-1}\Thp = \mx{\pa^{i(p-1)} & 0 \\  \pa^{-j}\theta^p -  \pa^{(p-1)i-j}\theta & \pa^{j(p-1)}}. \]
Evidently, we require $i,j\ge 0$. Furthermore, we need $v_K(\pa^{-j}\theta^p -  \pa^{(p-1)i-j}\theta ) \ge 0$; equivalently, 
\[v_K(\theta^{(p-1)i}-\pa^{(p-1)i}) \ge j-v_K(\theta).\]
Thus, $\theta^{(p-1)i} \equiv \pa^{(p-1)i} \bmod \pa^{j-v_K(\theta)}$, alternatively,
\[H_{i,j,\theta}=R[\pa^i t_1 + \theta t_2, \pa^j t_2],\; i,j,\ge 0,\;\theta \equiv \zeta\pa^i \bmod \pa^{\lfloor\frac{j-v_K(\theta)}{p-1}\rfloor}R,\; \zeta\in\mathbb{F}_p\subset R.\]
\end{example}
\begin{example}[\ref{mono}]
Finally, let $B$ be the cyclic permutation matrix
\[B=\mx{0&1\\1&0}. \]
Then the corresponding (monogenic) Hopf algebra is \[H=R[t_1,t_2]/(t_1^p-t_2,t_2^p-t_1) = R[t]/(t^{p^2}-t).\]

In this case,
\[\Th^{-1}B\Thp = \mx{\pa^{-i}\theta^p & \pa^{pj-i} \\ \pa^{pi-j}-\pa^{-(i+j)}\theta^{p+1} & -\pa^{(p-1)j-i}\theta}.\]

For this matrix to have coefficients in $R$, we need $pj \ge i$.  We also need
\begin{align*}
v_K(\theta) &\ge i/p \\
v_K(\theta) &\ge i-(p-1)j\\
v_K(\pa^{(p+1)i} - \theta^{p+1} ) &\ge i+j.
\end{align*}

In Example \ref{apn2too} we showed that $i/p \ge i-(p-1)j$ when $pj\ge i$, however this seems to be the only simplification possible. Thus the Hopf orders are
\[H_{i,j,\theta}=R[\pa^i t^p+\theta t, \pa^j t], \;  pj \ge i, \;v_K(\theta) \ge i/p,\; v_K(\pa^{(p+1)i} - \theta^{p+1} ) \ge i+j,\]
and $H_{i,j,\theta}$ is monogenic if and only if $v_K(\theta) = j$ and $pj=i$.

\end{example}

\bibliographystyle{amsalpha}
\bibliography{C:/Research/MyRefs}
\end{document}